\documentclass[12pt,reqno]{amsart}

\usepackage{amsmath,amssymb,amsthm,mathtools}
\usepackage{enumitem}
\usepackage{microtype}
\usepackage[colorlinks=true,linkcolor=blue,citecolor=blue,urlcolor=blue]{hyperref}

\newtheorem{theorem}{Theorem}[section]
\newtheorem{proposition}[theorem]{Proposition}
\newtheorem{lemma}[theorem]{Lemma}

\theoremstyle{remark}
\newtheorem{remark}[theorem]{Remark}

\numberwithin{equation}{section}

\newcommand{\R}{\mathbb R}
\newcommand{\D}{\mathbb D}
\newcommand{\Ree}{\operatorname{Re}}
\newcommand{\Imm}{\operatorname{Im}}
\newcommand{\e}{\mathrm e}
\newcommand{\ii}{\mathrm i}
\newcommand{\logtwo}{\log_2}
\newcommand{\logthree}{\log_3}

\begin{document}

\title[Large zeta sums]{Large zeta sums and zeros of the Riemann zeta function}

\author{Zikang Dong}
\author{Ruihua Wang}
\author{Weijia Wang}
\author{Hao Zhang}

\address[Zikang Dong]{School of Mathematical Sciences, Soochow University, Suzhou 215006, P. R. China}
\address[Ruihua Wang]{School of Fundamental Sciences, Hainan Bielefeld University of Applied Sciences, Danzhou 578101, P. R. China}
\address[Weijia Wang]{School of Mathematics, Shandong University, Jinan 250010, P. R. China}
\address[Hao Zhang]{School of Mathematics, Hunan University, Changsha 410082, P. R. China}

\email{zikangdong@gmail.com}
\email{ruih.wan9@gmail.com}
\email{weijiawang@amss.ac.cn}
\email{zhanghaomath@hnu.edu.cn}

\date{}

\begin{abstract}
For real $t$ and $x\ge 1$, set
\[
S(x,t)=\sum_{n\le x} n^{\ii t}.
\]
We prove an unconditional inverse theorem relating large values of $S(x,t)$, with $|t|$ large, to zeros of the Riemann zeta function near height $t$.  More precisely, if $T\le |t|\le 2T$, $\exp(\sqrt{\log T})\le x\le \sqrt T$, and $|S(x,t)|=x/N$ with $N\le (\log x)^{1/100}$, then for every $cN^6\le L\le (\log x)/2$ a disk centered at $1+\ii\phi$, where $|\phi-t|\ll N$, contains at least $L/360$ zeros of $\zeta(s)$.  As a consequence, a quantitative restriction on zeros in a short family of arbitrarily thin fixed windows to the left of the line $\Ree s=1$ yields $S(x,t)\ll x/(\log x)^{1/100}$ in polynomial ranges of $x$.  The proof adapts the zero-forcing mechanism of Granville and Soundararajan for large character sums.  In the zeta setting the spectral height is shifted by $t$, and an additional residue from the pole of $\zeta(s)$ appears in the Gaussian transform; in the range considered here that residue is exponentially small.
\end{abstract}


\maketitle

\section{Introduction}

For real $t$ and $x\ge1$, consider the exponential sum
\begin{equation}\label{eq:def-S}
S(x,t):=\sum_{n\le x}n^{\ii t}.
\end{equation}
The Dirichlet series associated with the completely multiplicative function $n\mapsto n^{\ii t}$ is
\[
\sum_{n=1}^{\infty}\frac{n^{\ii t}}{n^s}=\zeta(s-\ii t),
\qquad \Ree s>1.
\]
Thus the natural spectral height associated with a large value of $S(x,t)$ is not height zero, but height approximately $t$.  Our aim is to make this relation quantitative.

Sums of the form $\sum_{n\le x}n^{-\ii t}$ are already closely tied to the classical Lindel\"of hypothesis.  Gonek, Graham and Lee \cite{GGL20} recall the equivalent formulation that, for every fixed $B,\delta>0$,
\[
\sum_{n\le x}n^{-\ii t}
=\frac{x^{1-\ii t}}{1-\ii t}+O_{\delta,B}(x^{1/2}|t|^{\delta})
\qquad (1\le x\le |t|^B).
\]
In  earlier work of  part of the authors \cite{DWZ23}, we studied asymptotic approximations and large-value lower bounds for zeta sums in several ranges of $x$.  More recently, Yang \cite[Theorem~5]{Yang24} proved, under the Riemann hypothesis, a smooth-number approximation and a corresponding upper bound for these partial sums in a broad range of $x$ and $t$.  The typical and averaged sizes of zeta sums have also been studied by Harper \cite{Harper23}, while Gao \cite{Gao26} obtained conditional upper bounds for their moments.  These results concern direct upper bounds, moments, or large-value constructions.  The question considered here is different: we ask what the existence of one abnormally large sum forces about the zero set of $\zeta(s)$.  Thus the main contribution is the pointwise inverse implication, rather than the conditional upper bound by itself. Related analogues for sums of Fourier coefficients of cusp forms were studied by Lamzouri \cite{Lam}, and more recently by by Fréchette, Gerbelli-Gauthier, Hamieh and Tanabe \cite{FGHT}.

The argument is modeled on the work of Granville and Soundararajan on large character sums and zeros of Dirichlet $L$-functions \cite{GS18}.  Their method combines Hal\'asz's theorem and Lipschitz estimates for mean values of multiplicative functions with a Gaussian transform and a zero-repulsion estimate obtained from the Hadamard product.  For the present function $n^{\ii t}$ there are two changes that must be kept explicit.  First, the associated Dirichlet series is $\zeta(s-\ii t)$, so all zero neighborhoods are translated to height approximately $t$.  Second, $\zeta$ has a pole at $s=1$, and the corresponding residue appears when the contour in the Gaussian transform is shifted.

Our main result is the following zero-forcing theorem.

\begin{theorem}\label{thm:main}
There are absolute constants $c>0$ and $T_0\ge3$ with the following property.  Let $T\ge T_0$ and suppose that
\[
T\le |t|\le 2T,
\qquad
\exp(\sqrt{\log T})\le x\le \sqrt T.
\]
Assume that
\[
|S(x,t)|=\frac{x}{N},
\qquad
1\le N\le (\log x)^{1/100}.
\]
Then there is a real number $\phi$ satisfying
\[
|\phi-t|\le cN
\]
such that, for every parameter $L$ in the range
\[
cN^6\le L\le \frac12\log x,
\]
the disk
\begin{equation}\label{eq:main-disk}
\left\{s:\ |s-(1+\ii\phi)|<L\frac{\log T}{(\log x)^2}\right\}
\end{equation}
contains at least $L/360$ non-trivial zeros of $\zeta(s)$, counted with multiplicity.
\end{theorem}

The center in \eqref{eq:main-disk} is necessarily close to $1+\ii t$.  Here \(t_0\) plays the role of the small twisting parameter \(\phi\) in \cite[Theorem~1.3]{GS18}. Because the associated Dirichlet series is \(\zeta(s-it)\), the corresponding zero cluster is shifted to height \(t-t_0\), which is necessarily close to \(t\).

The most informative part of Theorem \ref{thm:main} is the regime in which the disk radius is genuinely smaller than the horizontal distance from the line $\Ree s=1$ to the critical line.  Writing $y_0=\log x$ and $Q=\log T$, this means
\[
L\frac{Q}{y_0^2}<\frac12.
\]
This subrange is non-empty whenever $y_0^2\gg N^6Q$.  It is precisely the regime used in Theorem \ref{thm:corollary}: the choice of $L$ there makes the radius at most $\delta$.  We therefore view the zero-forcing statement primarily as a local near-$1$ result, even though the theorem is valid for the full stated range of $L$.

As a consequence we obtain a flexible local zero criterion.  The parameter $\delta$ below measures the width of the zero window and may be any fixed positive number at most $1/4$.  Thus the hypothesis may be imposed in a strip arbitrarily close to the line $\Ree s=1$.

\begin{theorem}\label{thm:corollary}
There is an absolute constant $C>0$ with the following property.  Fix $A>0$ and $0<\delta\le1/4$.  There exists $T_0(A,\delta)\ge3$ such that, whenever $T\ge T_0(A,\delta)$,
\[
T\le |t|\le2T,
\qquad
\varepsilon>(\log T)^{-1/3},
\]
the following holds.  Suppose that, for every real $u$ with
\[
|u-t|\le C(\log T)^{1/100},
\]
the region
\begin{equation}\label{eq:zero-window}
\left\{s:\ \Ree s\ge1-\delta,\quad |\Imm s-u|\le\delta\right\}
\end{equation}
contains at most
\[
\frac{\delta\varepsilon^2\log T}{400}
\]
non-trivial zeros of $\zeta(s)$, counted with multiplicity.  Then, uniformly for
\[
T^\varepsilon\le x\le T^A,
\]
we have
\begin{equation}\label{eq:conditional-bound}
|S(x,t)|\ll_A \frac{x}{(\log x)^{1/100}}.
\end{equation}
\end{theorem}

Taking $\delta=1/4$ recovers the concrete window $\Ree s\ge3/4$, $|\Imm s-u|\le1/4$ and the zero threshold $\varepsilon^2\log T/1600$.  Under the Riemann hypothesis the hypothesis of Theorem \ref{thm:corollary} is automatic for every fixed $\delta<1/2$.  Stronger estimates for these partial sums are already available under RH through smooth-number methods; see \cite[Theorem~5]{Yang24}.  The relevance of Theorem \ref{thm:corollary} is therefore the locality of its zero hypothesis.

\begin{remark}\label{rem:density}
Combining Theorem \ref{thm:corollary} with a zero-density theorem gives exceptional-set statements.  For example, the zero-density estimate of Guth and Maynard \cite{GM26},
\[
N(\sigma,T)\ll T^{15(1-\sigma)/(3+5\sigma)+o(1)},
\]
shows, for any fixed $0<\delta\le1/4$, that the local hypothesis fails only on a set of $t\in[T,2T]$ of measure
\[
T^{15\delta/(8-5\delta)+o(1)}.
\]
We do not emphasize this as a main result: for the comparatively weak threshold in \eqref{eq:conditional-bound}, almost-all bounds can also be approached directly by standard high-moment methods for Dirichlet polynomials.  The distinctive content of Theorem \ref{thm:main} is the pointwise forcing of a local zero cluster from one large value of $S(x,t)$.
\end{remark}

It is important that some upper control on $x$ relative to $t$ is present in Theorem \ref{thm:corollary}.  Indeed, for fixed $t$, Euler summation gives
\begin{equation}\label{eq:euler-large-x}
S(x,t)=\frac{x^{1+\ii t}-1}{1+\ii t}+O(1+|t|\log x),
\end{equation}
so an estimate of the form $S(x,t)\ll x/(\log x)^{1/100}$ cannot hold for arbitrarily large $x$ with $t$ fixed.

The paper is organized as follows.  In Section \ref{sec:mean-values} we recall the required mean-value estimates and specialize them to $n^{\ii t}$.  Section \ref{sec:zeta-tools} develops the zeta-function estimates, including the Gaussian identity with its pole term.  In Section \ref{sec:main-proof} we prove Theorem \ref{thm:main}.  Finally, Section \ref{sec:cor-proof} deduces Theorem \ref{thm:corollary}.

\section{Mean values of multiplicative functions}\label{sec:mean-values}

We use the pretentious framework in the form developed in \cite{GS03,GS18}.  If $f$ and $g$ are multiplicative functions with $|f(n)|,|g(n)|\le1$, define
\[
\D(f,g;x)^2
:=\sum_{p\le x}\frac{1-\Ree(f(p)\overline{g(p)})}{p}.
\]
Let
\[
F(s)=\sum_{n=1}^{\infty}\frac{f(n)}{n^s},
\qquad \Ree s>1,
\]
and, for fixed $x$, choose $\phi=\phi_f(x)$ with $|\phi|\le\log x$ so that
\[
u\mapsto \left|F\left(1+\frac1{\log x}+\ii u\right)\right|
\]
is maximal at $u=\phi$.  Put
\[
M=M_f(x):=\D(f,n^{\ii\phi};x)^2,
\qquad
f_\phi(n):=f(n)n^{-\ii\phi}.
\]
We shall use the following standard estimates.  They are consequences of Hal\'asz's theorem, the mean-value comparison formula, and the Lipschitz estimate for multiplicative functions; see \cite[Section 2]{GS18} and \cite{GS03}.

\begin{lemma}\label{lem:standard-mean}
Let $f$ be multiplicative with $|f(n)|\le1$.  With the notation above, as $x\to\infty$,
\begin{equation}\label{eq:halasz}
\frac1x\left|\sum_{n\le x}f(n)\right|
\ll
\frac{(M+1)\e^{-M}}{1+|\phi|}
+(\log x)^{\sqrt3-2+o(1)}.
\end{equation}
Moreover,
\begin{align}
\frac1x\sum_{n\le x}f(n)
&=\frac{x^{\ii\phi}}{1+\ii\phi}\,
\frac1x\sum_{n\le x}f_\phi(n) \notag\\
&\quad+O\!\left(
\frac{\log(e+|\phi|)}{\log x}
\exp\!\left(\sum_{p\le x}\frac{|1-f_\phi(p)|}{p}\right)
\right),
\label{eq:mean-comparison}
\end{align}
and, uniformly for $\sqrt x\le z\le x^2$,
\begin{equation}\label{eq:lipschitz-standard}
\left|
\frac1x\sum_{n\le x}f_\phi(n)
-
\frac1z\sum_{n\le z}f_\phi(n)
\right|
\ll
\left(\frac{1+|\log x-\log z|}{\log x}\right)^{1-2/\pi+o(1)}.
\end{equation}
Here the $o(1)$ terms are uniform for the indicated range of $z$ and for $1$-bounded multiplicative functions.
\end{lemma}

\begin{proof}
The Hal\'asz estimate \eqref{eq:halasz} and the Lipschitz estimate
\eqref{eq:lipschitz-standard} are exactly the forms recorded in
\cite[Section~2, (2.3) and (2.5)]{GS18}, which in turn follow from
\cite[Theorems~2b and~4]{GS03}.

For completeness, the comparison formula deserves one clarification.  Apply
\cite[Lemma~7.1]{GS03} not to $f$ itself, but to
$g=f_\phi=f(n)n^{-\ii\phi}$ with the twisting parameter $\alpha=\phi$.
Since $g(n)n^{\ii\phi}=f(n)$, that lemma gives
\[
\sum_{n\le x}f(n)
=\frac{x^{\ii\phi}}{1+\ii\phi}\sum_{n\le x}f_\phi(n)
+O\!\left(
\frac{x\log(e+|\phi|)}{\log x}
\exp\!\left(\sum_{p\le x}\frac{|1-f_\phi(p)|}{p}\right)
\right),
\]
which is \eqref{eq:mean-comparison} after division by $x$.  Notice that
$|\phi|\le\log x$, so $\log(e+|\phi|)\ll\logtwo x$ whenever the latter
coarser form is convenient.
\end{proof}

We now specialize to
\[
f(n)=n^{\ii t},
\qquad
F(s)=\zeta(s-\ii t).
\]
For fixed $x$ and $t$, let $t_0=t_0(x,t)$ be a point with $|t_0|\le\log x$ at which
\begin{equation}\label{eq:t0-definition}
u\longmapsto
\left|\zeta\left(1+\frac1{\log x}+\ii(u-t)\right)\right|
\end{equation}
is maximal.  Put
\begin{equation}\label{eq:def-M}
M:=\sum_{p\le x}\frac{1-\Ree(p^{\ii(t-t_0)})}{p}.
\end{equation}

\begin{lemma}\label{lem:t0}
There is an absolute constant $x_0\ge3$ such that the following holds for $x\ge x_0$.  Suppose that
\[
|S(x,t)|=\frac{x}{N},
\qquad
1\le N\le(\log x)^{1/100}.
\]
Then
\begin{equation}\label{eq:t0-bound}
|t_0|\ll N
\end{equation}
and
\begin{equation}\label{eq:M-bound}
M\le \frac1{100}\logtwo x+\logthree x+O(1).
\end{equation}
Writing $y_0=\log x$ and $\tau=t-t_0$, we have, for all real $y$,
\begin{equation}\label{eq:lipschitz-y}
\left|
\frac{S(\e^y,\tau)}{\e^y}
-
\frac{S(\e^{y_0},\tau)}{\e^{y_0}}
\right|
\ll
\left(\frac{1+|y-y_0|}{y_0}\right)^{1/3},
\end{equation}
where $S(\e^y,\tau)=0$ when $\e^y<1$.  Finally,
\begin{equation}\label{eq:twisted-relation}
S(x,\tau)
=(1+\ii t_0)x^{-\ii t_0}S(x,t)
+O\!\left(\frac{x}{(\log x)^{3/4}}\right).
\end{equation}
\end{lemma}

\begin{proof}
Applying \eqref{eq:halasz} to $f(n)=n^{\ii t}$ and using the hypothesis gives
\[
\frac1N
\ll
\frac{(M+1)\e^{-M}}{1+|t_0|}
+(\log x)^{\sqrt3-2+o(1)}.
\]
Since $2-\sqrt3>1/100$ and the $o(1)$ in \eqref{eq:halasz} is uniform, there is an absolute $x_0$ such that for $x\ge x_0$ the second term is $o(1/N)$ uniformly in the present range of $N$.  Therefore
\begin{equation}\label{eq:M-key}
\frac{(M+1)\e^{-M}}{1+|t_0|}\gg\frac1N.
\end{equation}
As $(M+1)\e^{-M}\le1$, this proves \eqref{eq:t0-bound}.  It also follows from \eqref{eq:M-key} that
\[
(M+1)\e^{-M}\gg\frac1N,
\]
whence
\[
M\le \log N+\log\log(3N)+O(1),
\]
and \eqref{eq:M-bound} follows from $N\le(\log x)^{1/100}$.

If $|y-y_0|\le y_0/2$, then $z=\e^y$ lies in the range $\sqrt{x}\le z\le x^{3/2}$, so \eqref{eq:lipschitz-standard} applies.  Since the base
\[
\frac{1+|y-y_0|}{y_0}
\]
is at most $1/2+1/y_0<1$ for $y_0$ sufficiently large, and since $1-2/\pi>1/3$, the uniform $o(1)$ in \eqref{eq:lipschitz-standard} may be absorbed to give \eqref{eq:lipschitz-y}.  If $|y-y_0|>y_0/2$, the right hand side of \eqref{eq:lipschitz-y} is bounded below by a positive absolute constant, while each normalized partial sum has modulus at most $1$.  Enlarging the implied constant proves \eqref{eq:lipschitz-y} for all real $y$.

It remains to prove \eqref{eq:twisted-relation}.  By Cauchy--Schwarz, \eqref{eq:M-bound}, and Mertens' theorem,
\begin{align*}
\sum_{p\le x}\frac{|1-p^{\ii\tau}|}{p}
&\le
\left(\sum_{p\le x}\frac1p\right)^{1/2}
\left(\sum_{p\le x}\frac{|1-p^{\ii\tau}|^2}{p}\right)^{1/2}\\
&\le
\left(2M\sum_{p\le x}\frac1p\right)^{1/2}\\
&\le \left(\frac{\sqrt2}{10}+o(1)\right)\logtwo x
\le \frac17\logtwo x
\end{align*}
for sufficiently large $x$.  Formula \eqref{eq:mean-comparison} now gives
\[
S(x,t)=\frac{x^{\ii t_0}}{1+\ii t_0}S(x,\tau)
+O\!\left(
\frac{x\log(e+|t_0|)}{\log x}(\log x)^{1/7}
\right).
\]
Solving for $S(x,\tau)$ introduces a factor $|1+\ii t_0|\ll N$.  Since
$|t_0|\ll N\le(\log x)^{1/100}$, we also have
$\log(e+|t_0|)\ll\logtwo x$.  Therefore the resulting error is
\[
\ll x(\log x)^{1/100-6/7}\logtwo x
\ll \frac{x}{(\log x)^{3/4}},
\]
which proves \eqref{eq:twisted-relation}.
\end{proof}

\section{Zeta-function estimates and a Gaussian transform}\label{sec:zeta-tools}

We next record three lemmas for the Riemann zeta function.  Throughout, $\rho=\beta+\ii\gamma$ denotes a non-trivial zero of $\zeta(s)$, counted with multiplicity.

\begin{lemma}\label{lem:zeta-crude}
Uniformly for $0<\lambda\le1/2$ and $v\in\R$,
\begin{equation}\label{eq:zeta-crude}
|\zeta(1-\lambda+\ii v)|
\ll \frac{(2+|v|)^\lambda}{\lambda}.
\end{equation}
\end{lemma}

\begin{proof}
Put $s=1-\lambda+\ii v$, $\sigma=\Ree s=1-\lambda$, and $X=2+|v|$.  We use the Euler--Maclaurin continuation in the form
\begin{equation}\label{eq:euler-zeta-uniform}
\zeta(s)
=\sum_{n\le X}n^{-s}+\frac{X^{1-s}}{s-1}
+O\!\left(X^{-\sigma}+\frac{|s|}{\sigma}X^{-\sigma}\right),
\qquad \sigma>0,\quad s\ne1,
\end{equation}
uniformly after the pole term has been displayed explicitly.  For completeness, this follows by writing the remainder as a Stieltjes integral involving the bounded sawtooth function and integrating once by parts.  In our range $\sigma\ge1/2$ and $|s|\ll X$.  Moreover,
\[
\sum_{n\le X}n^{-\sigma}
\le 1+\int_1^X u^{-1+\lambda}\,du
\ll \frac{X^\lambda}{\lambda},
\]
while
\[
\left|\frac{X^{1-s}}{s-1}\right|
\le \frac{X^\lambda}{|s-1|}
\le \frac{X^\lambda}{\lambda}.
\]
Finally the error in \eqref{eq:euler-zeta-uniform} is $O(X^\lambda)$, since $|s|X^{-\sigma}\ll X^\lambda$.  As $0<\lambda\le1/2$, this is also $O(X^\lambda/\lambda)$, and \eqref{eq:zeta-crude} follows.
\end{proof}

\begin{lemma}\label{lem:zero-bound}
Uniformly for $0<\lambda\le1/2$ and $u\in\R$,
\begin{equation}\label{eq:zero-bound}
|\zeta(1-\lambda+\ii u)|
\ll \frac1\lambda
\exp\!\left(
\sum_\rho\frac{2\lambda^2}{|1+\lambda+\ii u-\rho|^2}
\right).
\end{equation}
\end{lemma}

\begin{proof}
For $|u|\le2$, the function $(s-1)\zeta(s)$ is bounded on a fixed neighborhood of the compact set $\{1/2\le\Ree s\le1,\ |\Imm s|\le2\}$.  Since
\[
|1-\lambda+\ii u-1|\ge\lambda,
\]
we obtain $|\zeta(1-\lambda+\ii u)|\ll1/\lambda$, and \eqref{eq:zero-bound} follows.  We may therefore assume $|u|>2$.

Set
\[
s_0=1+\lambda+\ii u,
\qquad
s_1=1-\lambda+\ii u,
\]
and recall the completed zeta function
\[
\xi(s)=\frac12s(s-1)\pi^{-s/2}\Gamma(s/2)\zeta(s).
\]
Its Hadamard factorization may be written
\[
\xi(s)=\e^{A+Bs}\prod_\rho
\left(1-\frac{s}{\rho}\right)\e^{s/\rho},
\]
where the product and sums over zeros are understood in the usual symmetric sense and
\[
\Ree B=-\sum_\rho\Ree\frac1\rho.
\]
Consequently the exponential factors cancel upon taking absolute values of the quotient at $s_1$ and $s_0$, and
\[
\left|\frac{\xi(s_1)}{\xi(s_0)}\right|
=
\prod_\rho\frac{|s_1-\rho|}{|s_0-\rho|}.
\]
Using the definition of $\xi$, the rational factors are $O(1)$ for $|u|>2$, while Stirling's formula gives
\[
\left|\frac{\Gamma(s_0/2)}{\Gamma(s_1/2)}\right|
\ll (2+|u|)^\lambda.
\]
Thus
\begin{equation}\label{eq:ratio-xi}
\left|\frac{\zeta(s_1)}{\zeta(s_0)}\right|
\ll
(2+|u|)^\lambda
\prod_\rho\frac{|s_1-\rho|}{|s_0-\rho|}.
\end{equation}
This is the zeta analogue of the first step in \cite[Lemma 3.1]{GS18}.

Since $\beta\le1$,
\begin{align*}
\frac{|s_1-\rho|}{|s_0-\rho|}
&=\left(1-\frac{4\lambda(1-\beta)}{|s_0-\rho|^2}\right)^{1/2}\\
&\le
\exp\!\left(-\frac{2\lambda(1-\beta)}{|s_0-\rho|^2}\right)\\
&=
\exp\!\left(
-2\lambda\Ree\frac1{s_0-\rho}
+\frac{2\lambda^2}{|s_0-\rho|^2}
\right).
\end{align*}
The logarithmic derivative of the same Hadamard product gives
\[
\sum_\rho\Ree\frac1{s_0-\rho}=\Ree\frac{\xi'}{\xi}(s_0).
\]
On the other hand, differentiating the defining formula for $\xi$ yields
\[
\frac{\xi'}{\xi}(s)
=
\frac1s+\frac1{s-1}-\frac12\log\pi
+\frac12\frac{\Gamma'}{\Gamma}(s/2)
+\frac{\zeta'}{\zeta}(s).
\]
Therefore Stirling's formula implies
\begin{equation}\label{eq:logder-lower}
\sum_\rho\Ree\frac1{s_0-\rho}
=\frac12\log(2+|u|)+\Ree\frac{\zeta'}{\zeta}(s_0)+O(1).
\end{equation}
By the Euler product,
\[
\left|\frac{\zeta'}{\zeta}(1+\lambda+\ii u)\right|
\le -\frac{\zeta'}{\zeta}(1+\lambda)
=\frac1\lambda+O(1).
\]
Hence
\[
\sum_\rho\Ree\frac1{s_0-\rho}
\ge\frac12\log(2+|u|)-\frac1\lambda-O(1).
\]
After insertion in \eqref{eq:ratio-xi}, the factor $(2+|u|)^\lambda$ is cancelled, up to an absolute constant, by the contribution of the term
\[
-2\lambda\sum_\rho\Ree\frac1{s_0-\rho}.
\]
We obtain
\[
\left|\frac{\zeta(s_1)}{\zeta(s_0)}\right|
\ll
\exp\!\left(
\sum_\rho\frac{2\lambda^2}{|s_0-\rho|^2}
\right).
\]
Finally $|\zeta(s_0)|\le\zeta(1+\lambda)\ll1/\lambda$, proving the lemma.
\end{proof}

The next identity is the zeta analogue of \cite[Lemma 3.2]{GS18}.  Unlike the Dirichlet $L$-function case, an additional residue occurs because $\zeta$ has a pole at $1$.

\begin{lemma}\label{lem:gaussian}
Let $\tau\in\R$, $\mathcal T>0$, and $0<\lambda\le1/2$.  Then
\begin{align}
&\sqrt{2\pi\mathcal T}
\int_{-\infty}^{\infty}
S(\e^y,\tau)
\exp\!\left((\lambda-1)y-\frac{\mathcal T y^2}{2}\right)\,dy \notag\\
&\quad=
\int_{-\infty}^{\infty}
\frac{\zeta(1-\lambda+\ii(\xi-\tau))}{1-\lambda+\ii\xi}
\exp\!\left(-\frac{\xi^2}{2\mathcal T}\right)\,d\xi \notag\\
&\qquad\quad+
\frac{2\pi}{1+\ii\tau}
\exp\!\left(\frac{(\lambda+\ii\tau)^2}{2\mathcal T}\right).
\label{eq:gaussian-transform}
\end{align}
\end{lemma}

\begin{proof}
We give a Fourier-transform proof, which also fixes the normalization of the pole term.  For a real parameter $\alpha<0$, set
\[
\Phi_\alpha(y)=S(\e^y,\tau)\e^{(\alpha-1)y}.
\]
Since $S(\e^y,\tau)=0$ for $y<0$ and $|S(\e^y,\tau)|\le \e^y$ for $y\ge0$, the function $\Phi_\alpha$ is integrable.  With the Fourier-transform convention
\[
\widehat \Phi_\alpha(\xi)
=\int_{-\infty}^{\infty}\Phi_\alpha(y)\e^{-\ii\xi y}\,dy,
\]
absolute convergence permits us to interchange the sum and the integral.  Hence
\begin{align*}
\widehat \Phi_\alpha(\xi)
&=\sum_{n\ge1}n^{\ii\tau}
  \int_{\log n}^{\infty}
  \e^{(\alpha-1-\ii\xi)y}\,dy\\
&=\frac{1}{1-\alpha+\ii\xi}
  \sum_{n\ge1}n^{-1+\alpha+\ii(\tau-\xi)}\\
&=\frac{\zeta(1-\alpha+\ii(\xi-\tau))}
        {1-\alpha+\ii\xi}.
\end{align*}
On the other hand,
\[
\sqrt{2\pi\mathcal T}\,
\e^{-\mathcal T y^2/2}
=
\int_{-\infty}^{\infty}
\e^{-\xi^2/(2\mathcal T)}\e^{-\ii\xi y}\,d\xi.
\]
A second application of Fubini therefore yields, for $\alpha<0$,
\begin{align}
F(\alpha)
&:=\sqrt{2\pi\mathcal T}
\int_{-\infty}^{\infty}
S(\e^y,\tau)
\exp\!\left((\alpha-1)y-\frac{\mathcal T y^2}{2}\right)dy \notag\\
&=\int_{-\infty}^{\infty}
\frac{\zeta(1-\alpha+\ii(\xi-\tau))}
     {1-\alpha+\ii\xi}
\e^{-\xi^2/(2\mathcal T)}\,d\xi.
\label{eq:gaussian-alpha-negative}
\end{align}
The left-hand side is an entire function of $\alpha$, because the Gaussian dominates uniformly on compact subsets of the $\alpha$-plane.  We now continue the right-hand side from $\alpha<0$ to $0<\alpha\le1/2$.  Its only singularity that crosses the real $\xi$-axis is the pole of the zeta factor at
\[
\xi=\tau-\ii\alpha.
\]
(The pole of $(1-\alpha+\ii\xi)^{-1}$ remains in the upper half-plane since $\alpha\le1/2$.)  Keeping the $\xi$-contour below the moving zeta pole and then straightening it back to the real axis gives
\[
F(\alpha)
=
\int_{-\infty}^{\infty}
\frac{\zeta(1-\alpha+\ii(\xi-\tau))}
     {1-\alpha+\ii\xi}
\e^{-\xi^2/(2\mathcal T)}\,d\xi
+2\pi\ii\operatorname*{Res}_{\xi=\tau-\ii\alpha}H_\alpha(\xi),
\]
where
\[
H_\alpha(\xi)=
\frac{\zeta(1-\alpha+\ii(\xi-\tau))}
     {1-\alpha+\ii\xi}
\e^{-\xi^2/(2\mathcal T)}.
\]
The Gaussian gives rapid decay on every contour occurring in this finite deformation.  Since the residue of $\zeta(1-\alpha+\ii(\xi-\tau))$ as a function of $\xi$ is $1/\ii=-\ii$, we have
\begin{align*}
2\pi\ii\operatorname*{Res}_{\xi=\tau-\ii\alpha}H_\alpha(\xi)
&=\frac{2\pi}{1+\ii\tau}
  \exp\!\left(-\frac{(\tau-\ii\alpha)^2}{2\mathcal T}\right)\\
&=\frac{2\pi}{1+\ii\tau}
  \exp\!\left(\frac{(\alpha+\ii\tau)^2}{2\mathcal T}\right).
\end{align*}
Taking $\alpha=\lambda$ proves \eqref{eq:gaussian-transform}.
\end{proof}

We shall also use the following standard consequence of the logarithmic derivative of $\xi(s)$.

\begin{lemma}\label{lem:zero-sum-upper}
Uniformly for $a>0$ and $|v|\ge2$,
\begin{equation}\label{eq:zero-sum-upper}
\sum_\rho\Ree\frac1{1+a+\ii v-\rho}
\le \frac12\log(2+a+|v|)+\frac1a+O(1).
\end{equation}
\end{lemma}

\begin{proof}
From the logarithmic derivative of the completed zeta function,
\[
\sum_\rho\Ree\frac1{s-\rho}
=
\Ree\frac{\zeta'}{\zeta}(s)
+\frac12\Ree\frac{\Gamma'}{\Gamma}(s/2)
+O(1)
\]
for $\Ree s>1$ and $|\Imm s|\ge2$; the rational factors $1/s$ and $1/(s-1)$ are absorbed in $O(1)$.  Stirling's formula, uniformly in the right half-plane, gives
\[
\Ree\frac{\Gamma'}{\Gamma}(s/2)
\le \log(2+|s|)+O(1).
\]
At $s=1+a+\ii v$, the Euler product implies
\[
\Ree\frac{\zeta'}{\zeta}(s)
\le\left|\frac{\zeta'}{\zeta}(s)\right|
\le-\frac{\zeta'}{\zeta}(1+a)
\le\frac1a+O(1).
\]
Since $|s|\ll 2+a+|v|$, this proves \eqref{eq:zero-sum-upper}.
\end{proof}

\section{Large zeta sums force many zeros}\label{sec:main-proof}

We first establish the analogue of \cite[Proposition 3.4]{GS18}.

\begin{proposition}\label{prop:forcing}
There are absolute constants $c_0>0$ and $T_1\ge3$ such that the following holds.  Let $T\ge T_1$,
\[
T\le |t|\le2T,
\qquad
\exp(\sqrt{\log T})\le x\le\sqrt T.
\]
Put $y_0=\log x$ and suppose that
\[
|S(x,t)|=\frac{x}{N},
\qquad
1\le N\le y_0^{1/100}.
\]
Let $t_0$ be defined by \eqref{eq:t0-definition}.  If
\begin{equation}\label{eq:lambda-range}
\frac{c_0N^6}{y_0}\le\lambda\le\frac12,
\end{equation}
then there is a real $\eta$ with
\begin{equation}\label{eq:eta-range}
|\eta|\le2\lambda\sqrt{\frac{\log T}{y_0}}
\end{equation}
such that
\begin{equation}\label{eq:forcing}
\sum_\rho
\frac{\lambda}{|1+\lambda+\ii(t-t_0+\eta)-\rho|^2}
\ge\frac{y_0}{4}.
\end{equation}
\end{proposition}

\begin{proof}
Set
\[
\tau=t-t_0,
\qquad
\mathcal T=\frac{\lambda}{y_0}.
\]
By Lemma \ref{lem:t0},
\[
\frac{S(\e^y,\tau)}{\e^y}
=
\frac{S(\e^{y_0},\tau)}{\e^{y_0}}
+O\!\left(\left(\frac{1+|y-y_0|}{y_0}\right)^{1/3}\right).
\]
Since
\[
\lambda y-\frac{\mathcal T y^2}{2}
=
\frac{\lambda y_0}{2}-\frac{\mathcal T}{2}(y-y_0)^2,
\]
a change of variables in the Gaussian error integral gives
\begin{align}
&\sqrt{2\pi\mathcal T}
\int_{-\infty}^{\infty}
\frac{S(\e^y,\tau)}{\e^y}
\exp\!\left(\lambda y-\frac{\mathcal T y^2}{2}\right)dy \notag\\
&\qquad=
2\pi\e^{\lambda y_0/2}
\left(
\frac{S(\e^{y_0},\tau)}{\e^{y_0}}
+O((\lambda y_0)^{-1/6})
\right).
\label{eq:gaussian-left-eval}
\end{align}
By \eqref{eq:twisted-relation},
\[
\frac{S(\e^{y_0},\tau)}{\e^{y_0}}
=(1+\ii t_0)\e^{-\ii t_0y_0}\frac{S(x,t)}x
+O(y_0^{-3/4}).
\]
The main term has modulus $\sqrt{1+t_0^2}/N\ge1/N$.  By choosing the absolute constant $c_0$ in the proposition sufficiently large, $(\lambda y_0)^{-1/6}\le c_0^{-1/6}/N$, while
\[
y_0^{-3/4}\le \frac{y_0^{-3/4+1/100}}{N}=\frac{y_0^{-37/50}}{N}=o(1/N).
\]  Hence \eqref{eq:gaussian-left-eval} yields
\begin{equation}\label{eq:gaussian-lower}
\left|
\sqrt{2\pi\mathcal T}
\int_{-\infty}^{\infty}
\frac{S(\e^y,\tau)}{\e^y}
\exp\!\left(\lambda y-\frac{\mathcal T y^2}{2}\right)dy
\right|
\ge\frac{\pi}{N}\e^{\lambda y_0/2}.
\end{equation}

Apply Lemma \ref{lem:gaussian}.  Since $|t_0|\ll N$ by Lemma \ref{lem:t0} and $T\le|t|\le2T$, we have $|\tau|\asymp T$.  The residue in \eqref{eq:gaussian-transform} has modulus
\begin{equation}\label{eq:residue-small}
\ll \frac1T
\exp\!\left(
\frac{\lambda y_0}{2}-\frac{\tau^2y_0}{2\lambda}
\right),
\end{equation}
and is therefore negligible compared with \eqref{eq:gaussian-lower}.  Consequently
\begin{equation}\label{eq:zeta-integral-lower}
\frac1N\e^{\lambda y_0/2}
\ll
\int_{-\infty}^{\infty}
\frac{|\zeta(1-\lambda+\ii(\xi-\tau))|}{|1-\lambda+\ii\xi|}
\exp\!\left(-\frac{\xi^2}{2\mathcal T}\right)d\xi.
\end{equation}
It follows that
\begin{equation}\label{eq:weighted-maximum}
\max_{\xi\in\R}
\frac{|\zeta(1-\lambda+\ii(\xi-\tau))|}{|1-\lambda+\ii\xi|}
\exp\!\left(-\frac{\xi^2}{4\mathcal T}\right)
\gg
\sqrt{\frac{y_0}{\lambda}}\,
\frac{\e^{\lambda y_0/2}}N.
\end{equation}

Put
\[
B=2\lambda\sqrt{\frac{\log T}{y_0}}.
\]
Suppose first that $B<|\xi|\le T/2$.  Since $|\tau|\asymp T$, we have $|\xi-\tau|\asymp T$.  Lemma \ref{lem:zeta-crude} and the definition of $B$ give
\begin{align*}
&\frac{|\zeta(1-\lambda+\ii(\xi-\tau))|}{|1-\lambda+\ii\xi|}
\exp\!\left(-\frac{\xi^2}{4\mathcal T}\right)\\
&\qquad\ll
\frac{T^\lambda}{\lambda}
\exp(-\lambda\log T)
\ll\frac1\lambda.
\end{align*}
For $|\xi|>T/2$, Lemma \ref{lem:zeta-crude} gives
\[
|\zeta(1-\lambda+\ii(\xi-\tau))|
\ll \frac{(2+|\xi|+T)^\lambda}{\lambda}.
\]
Since $|1-\lambda+\ii\xi|\ge1/2$, it remains to control
\[
(2+|\xi|+T)^\lambda
\exp\!\left(-\frac{\xi^2y_0}{4\lambda}\right).
\]
For $|\xi|\ge T/2$ its logarithm is at most
\[
\lambda\log(2+|\xi|+T)-\frac{\xi^2y_0}{4\lambda},
\]
which is negative with magnitude $\gg T^2y_0/\lambda$ at $|\xi|=T/2$ and decreases thereafter once the absolute threshold $T_1$ is enlarged if necessary.  Thus this range is $O(1/\lambda)$, in fact exponentially smaller.

On the other hand, the right hand side of \eqref{eq:weighted-maximum} is
\[
\frac1\lambda
\frac{\sqrt{\lambda y_0}}{N}\e^{\lambda y_0/2},
\]
which is larger than any prescribed fixed multiple of $1/\lambda$ once the absolute constant $c_0$ is chosen sufficiently large.  Thus the maximum in \eqref{eq:weighted-maximum} is attained at a point $\xi$ with $|\xi|\le B$.  At this point, using $\lambda y_0\ge c_0N^6$ once more,
\begin{equation}\label{eq:large-zeta-point}
|\zeta(1-\lambda+\ii(\xi-\tau))|
\gg
\frac1\lambda(\lambda y_0)^{1/3}\e^{\lambda y_0/2}.
\end{equation}
Let
\[
\mathcal Z:=\sum_\rho
\frac{\lambda}{|1+\lambda+\ii(\xi-\tau)-\rho|^2}.
\]
Lemma \ref{lem:zero-bound} and \eqref{eq:large-zeta-point} imply
\[
(\lambda y_0)^{1/3}\e^{\lambda y_0/2}
\ll \exp(2\lambda\mathcal Z).
\]
Taking logarithms,
\[
2\lambda\mathcal Z
\ge \frac{\lambda y_0}{2}
+\frac13\log(\lambda y_0)-O(1).
\]
Since $\lambda y_0\ge c_0N^6\ge c_0$, choosing the absolute constant $c_0$ sufficiently large makes the last two terms non-negative.  Hence
\[
\mathcal Z\ge\frac{y_0}{4},
\]
which is the required weighted zero lower bound.  Finally, the non-trivial zeros of $\zeta$ are invariant under complex conjugation.  Replacing $\rho$ by $\overline\rho$ and setting $\eta=-\xi$ gives \eqref{eq:eta-range} and \eqref{eq:forcing}.
\end{proof}

\begin{proof}[Proof of Theorem \ref{thm:main}]
Let
\[
y_0=\log x,
\qquad
\phi=t-t_0,
\qquad
\lambda=\frac{L}{40y_0}.
\]
Since $L\le y_0/2$, we have $\lambda\le1/80$.  Fix the absolute constant $c$ in the statement of the theorem so that $c\ge40c_0$ and also so that it dominates the absolute implied constant in \eqref{eq:t0-bound}.  Then $L\ge cN^6$ implies
\[
\lambda=\frac{L}{40y_0}\ge\frac{c_0N^6}{y_0},
\]
so Proposition \ref{prop:forcing} applies (after enlarging the absolute threshold $T_0$ to dominate $T_1$ and the finitely many absolute thresholds used in Section \ref{sec:mean-values}).  Thus for some $\eta$ satisfying \eqref{eq:eta-range},
\begin{equation}\label{eq:zero-weight-main}
\sum_\rho
\frac{\lambda}{|1+\lambda+\ii(\phi+\eta)-\rho|^2}
\ge\frac{y_0}{4}.
\end{equation}
Also, Lemma \ref{lem:t0} gives
\[
|\phi-t|=|t_0|\ll N.
\]

Write $Q=\log T$ and set
\[
a=20\lambda\frac{Q}{y_0}.
\]
Consider a zero $\rho$ satisfying
\begin{equation}\label{eq:outer-zero}
|1+\ii\phi-\rho|\ge40\lambda\frac{Q}{y_0}=2a.
\end{equation}
Since $\Ree\rho\le1$,
\[
|1+a+\ii\phi-\rho|\ge |1+\ii\phi-\rho|\ge2a.
\]
Moreover, $y_0\le Q/2$, and hence
\[
\frac{|\eta|}{a}
\le\frac1{10}\sqrt{\frac{y_0}{Q}}
\le\frac1{10\sqrt2}.
\]
By the triangle inequality and the preceding lower bound for $|1+a+\ii\phi-\rho|$,
\begin{align*}
|1+\lambda+\ii(\phi+\eta)-\rho|
&\ge |1+a+\ii\phi-\rho|-a-|\eta|\\
&\ge\frac9{20}|1+a+\ii\phi-\rho|.
\end{align*}
Therefore the total contribution to \eqref{eq:zero-weight-main} from zeros satisfying \eqref{eq:outer-zero} is at most
\begin{align}
5\lambda\sum_\rho\frac1{|1+a+\ii\phi-\rho|^2}
&\le \frac{5\lambda}{a}
\sum_\rho\Ree\frac1{1+a+\ii\phi-\rho} \notag\\
&=\frac{y_0}{4Q}
\sum_\rho\Ree\frac1{1+a+\ii\phi-\rho}.
\label{eq:outer-contribution}
\end{align}
Here we used $\Ree\rho\le1$ in the first inequality.

Since $|\phi-t|\ll N$ and $T\le|t|\le2T$, we have $|\phi|\asymp T$.  Moreover, $\lambda\le1/80$ and $y_0\ge\sqrt Q$, so
\[
a=20\lambda\frac{Q}{y_0}\le\frac{Q}{4y_0}\le\frac14\sqrt Q=o(T).
\]
Thus
\[
\log(2+a+|\phi|)=Q+O(1).
\]
Lemma \ref{lem:zero-sum-upper} therefore gives
\begin{equation}\label{eq:logsumbound}
\sum_\rho\Ree\frac1{1+a+\ii\phi-\rho}
\le\frac12Q+\frac1a+O(1).
\end{equation}
Now
\[
\frac1a=\frac{y_0}{20\lambda Q}.
\]
Because $\lambda\ge c_0N^6/y_0$ and $y_0\le Q/2$,
\[
\frac1a\le\frac{Q}{80c_0N^6}.
\]
Since $c_0$ was chosen sufficiently large, in particular $c_0\ge1$, the term $1/a$ is at most $Q/80$.  After enlarging the absolute threshold $T_0$ to absorb the $O(1)$ term, \eqref{eq:logsumbound} is at most $5Q/9$.  Consequently \eqref{eq:outer-contribution} is at most
\[
\frac{5y_0}{36}.
\]
By \eqref{eq:zero-weight-main}, the zeros satisfying
\[
|1+\ii\phi-\rho|<40\lambda\frac{Q}{y_0}
\]
therefore contribute at least
\[
\frac{y_0}{4}-\frac{5y_0}{36}=\frac{y_0}{9}.
\]
Each individual zero contributes at most $1/\lambda$, since
\[
|1+\lambda+\ii(\phi+\eta)-\rho|\ge1+\lambda-\Ree\rho\ge\lambda.
\]
Hence the number of zeros in this disk is at least
\[
\frac{\lambda y_0}{9}=\frac{L}{360}.
\]
Finally,
\[
40\lambda\frac{Q}{y_0}
=L\frac{\log T}{(\log x)^2},
\]
which is exactly the radius in \eqref{eq:main-disk}.  This proves the theorem.
\end{proof}

\section{A local zero-density consequence}\label{sec:cor-proof}

For the part of Theorem \ref{thm:corollary} with $x\ge\sqrt T$, a classical exponential-sum estimate is sufficient.  We record a convenient form.  The exponent pair $(1/6,2/3)$, obtained by the van der Corput method, gives the required power saving; see, for example, \cite[Chapter~8, especially Sections~8.3--8.4]{IK}.

\begin{lemma}\label{lem:large-x}
There is an absolute $T_2\ge3$ such that, uniformly for $T\ge T_2$, $T\le|t|\le2T$, and
\[
x\ge\sqrt T,
\]
one has
\begin{equation}\label{eq:large-x-bound}
S(x,t)\ll xT^{-1/13},
\end{equation}
with an absolute implied constant.
\end{lemma}

\begin{proof}
By complex conjugation it is enough to consider $t>0$.  Write
\[
e(z)=e^{2\pi\ii z},
\qquad
f(u)=\frac{t}{2\pi}\log u,
\qquad n^{\ii t}=e(f(n)).
\]
For a dyadic interval $X<n\le2X$ with $1\le X\le t$, the exponent pair
$(1/6,2/3)$ gives
\begin{equation}\label{eq:dyadic-exponent-pair}
\sum_{X<n\le2X}n^{\ii t}
\ll t^{1/6}X^{1/2}.
\end{equation}
Indeed, on this interval
$|f^{(j)}(u)|\asymp_j tX^{-j}$, so the standard exponent-pair estimate
applies with parameter $t/X\ge1$.  Summing over dyadic intervals is a
geometric sum.  Therefore, for $\sqrt T\le x\le t$,
\[
S(x,t)\ll t^{1/6}x^{1/2}
\ll xT^{-1/12}
\ll xT^{-1/13}.
\]

It remains to treat $x>t$.  The preceding estimate at $x=t$ gives
\[
S(t,t)\ll t^{2/3}.
\]
Now let $X\ge t$.  On $[X,2X]$ the derivative
\[
f'(u)=\frac{t}{2\pi u}
\]
is monotone and satisfies
\[
0<f'(u)\le\frac1{2\pi}<\frac12,
\qquad
\|f'(u)\|=f'(u)\ge \frac{t}{4\pi X},
\]
where $\|\cdot\|$ denotes distance to the nearest integer.  The
Kusmin--Landau first-derivative lemma therefore yields
\begin{equation}\label{eq:kusmin-large-x}
\sum_{X<n\le2X}n^{\ii t}\ll \frac{X}{t}.
\end{equation}
Summing \eqref{eq:kusmin-large-x} over the dyadic intervals between $t$
and $x$ gives $O(x/t)$.  Hence, using $t\asymp T$ and $x\ge t$,
\[
S(x,t)
\ll t^{2/3}+\frac{x}{t}
\ll xT^{-1/3}+xT^{-1}
\ll xT^{-1/13}.
\]
All implied constants here are absolute, and no upper restriction on $x$
is used.
\end{proof}

\begin{proof}[Proof of Theorem \ref{thm:corollary}]
Suppose first that
\[
\sqrt T\le x\le T^A.
\]
By Lemma \ref{lem:large-x},
\[
|S(x,t)|\ll xT^{-1/13}.
\]
Since $\log x\le A\log T$, we have
\[
T^{-1/13}\ll_A (\log x)^{-1/100}
\]
for $T\ge T_0(A,\delta)$ after increasing $T_0(A,\delta)$ if necessary.  Hence
\[
|S(x,t)|\ll_A \frac{x}{(\log x)^{1/100}}.
\]

It remains to consider
\begin{equation}\label{eq:small-x-range}
T^\varepsilon\le x\le\sqrt T.
\end{equation}
If this interval is non-empty, then $\varepsilon\le1/2$.  Moreover,
\[
\log x\ge\varepsilon\log T>(\log T)^{2/3},
\]
so $x\ge\exp(\sqrt{\log T})$ once the threshold $T_0(A,\delta)$ is enlarged if necessary.

Assume, for a contradiction, that
\[
|S(x,t)|\ge\frac{x}{(\log x)^{1/100}}.
\]
Write
\[
|S(x,t)|=\frac{x}{N}.
\]
Then
\[
1\le N\le(\log x)^{1/100},
\]
so Theorem \ref{thm:main} is applicable once an admissible $L$ is chosen.  Take
\begin{equation}\label{eq:L-choice}
L=\delta\varepsilon^2\log T.
\end{equation}
Since $\delta>0$ is fixed and $\varepsilon>(\log T)^{-1/3}$,
\[
L>\delta(\log T)^{1/3}.
\]
On the other hand,
\[
N^6\le(\log x)^{6/100}\le(\log T)^{6/100}.
\]
As $1/3>6/100$, we have $L\ge cN^6$ for all $T$ beyond a threshold depending at most on $\delta$ (which may be absorbed into $T_0(A,\delta)$), where $c$ is the constant in Theorem \ref{thm:main}.  Also, from \eqref{eq:small-x-range},
\[
\log x\ge\varepsilon\log T,
\]
and therefore
\[
L=\delta\varepsilon^2\log T
\le\delta\varepsilon\log x
\le\frac18\log x<\frac12\log x,
\]
Thus \eqref{eq:L-choice} lies in the admissible range of Theorem \ref{thm:main}.

That theorem produces a real number $\phi$ with
\[
|\phi-t|\le cN
\le c(\log x)^{1/100}
\le C(\log T)^{1/100},
\]
where the absolute constant $C$ in the statement is chosen once and for all to dominate $c$, such that the disk
\[
\left\{s:\ |s-(1+\ii\phi)|<L\frac{\log T}{(\log x)^2}\right\}
\]
contains at least $L/360$ zeros of $\zeta$.  Its radius is
\[
L\frac{\log T}{(\log x)^2}
=
\frac{\delta\varepsilon^2(\log T)^2}{(\log x)^2}
\le\delta.
\]
Hence the disk is contained in
\[
\left\{s:\ \Ree s\ge1-\delta,\quad |\Imm s-\phi|\le\delta\right\}.
\]
By the hypothesis of the theorem, the latter region contains at most
\[
\frac{\delta\varepsilon^2\log T}{400}
\]
zeros.  But Theorem \ref{thm:main} gives at least
\[
\frac{L}{360}
=
\frac{\delta\varepsilon^2\log T}{360}
>
\frac{\delta\varepsilon^2\log T}{400},
\]
a contradiction.  This proves \eqref{eq:conditional-bound}.
\end{proof}

\begin{remark}
The restriction $x\le T^A$ in Theorem \ref{thm:corollary} is needed only when the power saving of Lemma \ref{lem:large-x} is converted into a negative power of $\log x$; Lemma \ref{lem:large-x} itself has no upper restriction on $x$.  Some upper control on $x$ is nevertheless essential for the logarithmic conclusion.  Formula \eqref{eq:euler-large-x} shows that no estimate $S(x,t)=o(x)$ can hold uniformly for all arbitrarily large $x$ once $t$ is fixed.  If one only wants the range directly controlled by Theorem \ref{thm:main}, then Theorem \ref{thm:corollary} may instead be stated simply for $T^\varepsilon\le x\le\sqrt T$, in which case Lemma \ref{lem:large-x} is unnecessary.
\end{remark}

\section*{Acknowledgements}
The author acknowledges the use of OpenAI’s ChatGPT during the preparation of this
manuscript. Z. Dong is supported by  the National
	Natural Science Foundation of China (Grant No. 	1240011770). W. Wang is supported by the National
	Natural Science Foundation of China (Grant No. 1250012812). H. Zhang is supported by the Fundamental Research Funds for the Central Universities (Grant No. 531118010622), the National
	Natural Science Foundation of China (Grant No. 1240011979) and the Hunan Provincial Natural Science Foundation of China (Grant No. 2024JJ6120).

\end{document}